\documentclass[11pt]{amsart}
\usepackage{amsmath,amssymb,amsfonts,amsthm}
\newtheorem{theorem}{Theorem}[section]
\newtheorem{maintheorem}{Theorem}

\newtheorem{secondtheorem}{Theorem}

\newtheorem{Thirdtheorem}{Theorem}

\newtheorem{maincorollary}[maintheorem]{Corollary}
\newtheorem{secondcorollary}[secondtheorem]{Corollary}
\newtheorem{Thirdcorollary}[Thirdtheorem]{Corollary}

\newtheorem{corollary}[theorem]{Corollary}
\newtheorem{lemma}[theorem]{Lemma}
\theoremstyle{definition}

\newtheorem{example}[theorem]{Example}
\numberwithin{equation}{section}

\newcommand{\sing}{\textsf{Sing}}

\newcommand{\bb}{\text{BB}}
\newcommand{\cs}{\text{CS}}
\newcommand{\im}{\text{Im}}

\newcommand{\tr}{\text{Tr}}
\newcommand{\res}{\text{Res}}

\newcommand{\mc}[1]{\mathcal{#1}}

\begin{document}
\title[Dicritical singularities of Levi-flat hypersurfaces and foliations]{Existence of dicritical singularities of Levi-flat hypersurfaces and holomorphic foliations}
\author{Andr\'es Beltr\'an}
\address[A. Beltr\'an]{Dpto. Ciencias - Secci\'on Matem\'aticas, Pontificia Universidad Cat\'olica del Per\'u.}
\curraddr{Av. Universitaria 1801, San Miguel, Lima 32, Peru}
\email{abeltra@pucp.pe}
\author{Arturo Fern\'andez-P\'erez}
\address[A. Fern\'andez-P\'erez]{Departamento de Matem\'atica - ICEX, Universidade Federal de Minas Gerais, UFMG}
\curraddr{Av. Ant\^onio Carlos 6627, 31270-901, Belo Horizonte-MG, Brasil.}
\email{fernandez@ufmg.br}

\author{Hern\'an Neciosup}
\address[H. Neciosup]{Dpto. Ciencias - Secci\'on Matem\'aticas, Pontificia Universidad Cat\'olica del Per\'u.}
\curraddr{Av. Universitaria 1801, San Miguel, Lima 32, Peru}
\email{hneciosup@pucp.pe}
\thanks{This work is supported by the Pontificia Universidad Cat\'olica del Per\'u project VRI-DGI 2016-1-0018. Second author is partially supported by CNPq grant number 301825/2016-5}

\begin{abstract}
We study holomorphic foliations tangent to singular real-analytic Levi-flat hypersurfaces in compact complex manifolds of complex dimension two. We give some hypotheses to guarantee the existence of dicritical singularities of these objects.  As consequence, we give some applications to holomorphic foliations tangent to real-analytic Levi-flat hypersurfaces with singularities in $\mathbb{P}^2$.
\end{abstract}
\maketitle
\section{Introduction}
\par In this paper we study holomorphic foliations tangent to singular real-analytic Levi-flat hypersurfaces in compact complex manifolds of complex dimension two with emphasis on the type of singularities of them. Singular Levi-flat hypersurfaces in complex manifolds appear in many contexts, for example the zero set of the real-part of a holomorphic function or as an invariant set of a holomorphic foliation. While singular Levi-flat hypersurfaces have many properties of complex subvarieties, they have a more complicated geometry and inherit many pathologies from general real-analytic subvarieties. The interconnection between singular Levi-flat hypersurfaces 
and holomorphic foliations have been studied by many authors, see for example \cite{burns}, \cite{brunella}, \cite{alcides}, \cite{normal}, \cite{arturo}, \cite{generic}, \cite{arnold}, \cite{libro}, \cite{lebl}, \cite{singularlebl} and \cite{shafikov}. 
\par Let $M$ be a real-analytic closed subvariety of real dimension 3 in a compact complex manifold $X$ of complex dimension two. Unless specifically stated, subvarieties
are analytic, not necessarily algebraic. Throughout the text, the term \textit{real-analytic hypersurface} will be employed with the
meaning \textit{real-analytic subvariety of real dimension 3}.
Let us denote $M ^{*}$ the set of points of $M$ near which $M$ is a nonsingular real-analytic hypersurface. $M$ is said to be {\textit{Levi-flat}} if the codimension one distribution on $M^{*}$
$$T^{\mathbb{C}}M^{*}=TM^{*}\cap J(TM^{*})\subset TM^{*}$$
is integrable, in Frobenius sense. It follows that $M^{*}$ is foliated locally by immersed one-dimensional complex manifolds, the foliation defined by $T^{\mathbb{C}}M^{*}$ is called the \textit{Levi-foliation} and will be denoted by $\mathcal{L}$. 
 \par Let $\{U_j\}_{j\in I}$ be an open covering of $X$. A \textit{holomorphic foliation} $\mc{F}$ on $X$ can be described by a collection of holomorphic 1-forms $\omega_j\in\Omega^{1}_{X}(U_{j})$ with isolated zeros such that 
\begin{equation*}
\omega_i=g_{ij}\omega_j\,\,\,\,\,\,\,\,\,\text{on}\,\,\,U_i\cap U_j,\,\,\,\,\,\,\,\,\,\,\,\,\,\,g_{ij}\in\mathcal{O}^{*}_{X}(U_i\cap U_j).
\end{equation*}
The cocycle $\{g_{ij}\}$ defines a line bundle $N_{\mc{F}}$ on $X$, called \textit{normal bundle} of $\mc{F}$.  The \textit{singular set} $\sing(\mc{F})$ of $\mc{F}$ is the finite subset of $X$ defined by 
$$\sing(\mc{F})\cap U_{j}=\text{zeros of}\,\,\,\omega_{j},\,\,\,\,\,\,\,\,\,\,\,\,\,\,\,\forall j\in I.$$
A point $q\not\in\sing(\mc{F})$ is said to be \textit{regular}. 
We will be denote $c_1(N_{\mc{F}})\in H^2(X,\mathbb{Z})$ the \textit{first Chern class} of $N_{\mc{F}}$ and if $\Omega$ is a smooth closed 2-form on $X$ which represents, in the De Rham sense, the first Chern class of $N_{\mc{F}}$, we will use the following notation
$$c_1^{2}(N_{\mc{F}}):=\int_{X}\Omega\wedge\Omega.$$ 
We shall say that a holomorphic foliation $\mathcal{F}$ on $X$ is \textit{tangent} to $M$ if locally the leaves of the Levi foliation $\mathcal{L}$ on $M^{*}$ are also
leaves of $\mathcal{F}$. A singular point $p\in M$ is called \textit{dicritical} if for every neighborhood $U$ of $p$, infinitely many leaves of the Levi-foliation on $M^{*}\cap U$ have $p$ in their closure. Analogously, a singular point $p\in\sing(\mc{F})$ of a holomorphic foliation $\mc{F}$ is called \textit{dicritical} if for every neighborhood $U$ of $p$,  infinitely many leaves have $p$ in their closure. Otherwise it is called \textit{non-dicritical}. 
\par Recently dicritical singularities of singular real-analytic Levi-flat hypersurfaces have been characterized in terms of the \textit{Segre varieties}, see for instance Pinchuk-Shafikov-Sukhov \cite{pinchuk}. Using this characterization, the notion of dicritical singularity in the theory of holomorphic foliations coincides with the notion of Segre-degenerate singularity of a real-analytic Levi-flat hypersurface and therefore, we can use results of residue-type indices associated to singular points of holomorphic foliations \cite{index}, \cite{suwa} to prove the following result. 
\begin{maintheorem}\label{main_theorem}
Let $\mc{F}$ be a holomorphic foliation on a compact complex manifold $X$ of complex dimension two tangent to an irreducible real-analytic Levi-flat hypersurface $M \subset X$ such that $\sing(\mc{F})\subset M$. Suppose $c^2_1(N_\mc{F})>0$. Then there exists a dicritical singularity $p \in M$ such that $\mc
{F}$ has a non-constant meromorphic first integral at $p$. 
\end{maintheorem}
\par We emphasize that the existence of dicritical singularities of $M$ depends on the condition $c^2_1(N_\mc{F})>0$, because otherwise the result is false. In section \ref{examples_paper} we give some examples that show the importance of the condition $c^2_1(N_\mc{F})>0$. The condition $\sing(\mc{F})\subset M$ is used in the proof of Theorem \ref{main_theorem}, we do not know if this condition can be removed. 
\par In the sequel we apply Theorem \ref{main_theorem} to non-dicritical projective foliations, that is, holomorphic foliations on $\mathbb{P}^2$ with only non-dicritical singularities. 
 \begin{maincorollary}
Let $\mc{F}$ be a holomorphic foliation on $\mathbb{P}^2$ with only non-dicritical singularities. Then there are no singular real-analytic Levi-flat hypersurfaces in $\mathbb{P}^{2}$ tangent to $\mc{F}$ that contain $\sing(\mathcal{F})$.
\end{maincorollary}     
 
Now we apply Theorem \ref{main_theorem} for holomorphic foliations of degree 2  on $\mathbb{P}^2$ with a unique singularity.
\begin{secondcorollary}
Let $\mc{F}$ be a holomorphic foliation of degree 2 on $\mathbb{P}^2$ with a unique singular point $p$. Suppose that $\mc{F}$ is tangent to a singular real-analytic Levi-flat hypersurface $M\subset\mathbb{P}^2$ and $p\in M$. Then, up to automorphism,  $\mc{F}$ is given in affine coordinates $(x,y)\in\mathbb{C}^2$ by the 1-form
$$\omega=x^2dx+y^2(xdy-ydx).$$
Moreover, let $[x:y:z]$ be the homogenous coordinates of $\mathbb{P}^2$, then  $R=\frac{y^3-3x^2z}{3x^3}$ is a rational first integral for $\mc{F}$. 

\end{secondcorollary}
\par We say that a singularity $p$ of a germ of a real-analytic Levi-flat hypersurface $M$ is said to be \textit{semialgebraic}, if the germ of $M$ at $p$ is biholomorphic to a semialgebraic Levi-flat hypersurface. We recall that a real-analytic Levi-flat hypersurface is said to be \textit{semialgebraic}, if it is contained in a codimension one real-analytic subvariety defined by the vanishing of a real polynomial.   
\par To continue we apply Theorem \ref{main_theorem} to find \textit{semialgebraic} singularities of singular real-analytic Levi-flat hypersurfaces which are tangent to singular holomorphic foliations on compact complex manifolds of complex dimension two. Similarly results of algebraization of singularities of holomorphic foliations can be found in \cite{genzmer}. Recently in \cite{casale}, Casale considered the algebraization problem for \textit{simple dicritical singularities} of germs of holomorphic foliations. These singularities are those that become nonsingular after one blow-up and such that a unique leaf is tangent to the exceptional divisor with tangency order of one. Motived by \cite{casale}, we state the following result.

\begin{Thirdcorollary}
Let $\mc{F}$ be a holomorphic foliation on a compact complex manifold $X$ of complex dimension two tangent to an irreducible real-analytic Levi-flat hypersurface $M \subset X$ such that $\sing(\mc{F})\subset M$. Suppose that $c^2_1(N_{\mc{F}})>0$ and $\mc{F}$ has only a unique simple dicritical singularity $p\in M$. Then there exists an algebraic surface $V$, a rational
function $H$ on $V$ and a point $q\in V$ such that the germ of $M$  at $p$ is biholomorphic to a semialgebraic Levi-flat hypersurface $M'\subset V$ in a neighborhood of $q$.
\end{Thirdcorollary}

\par Finally we state a result that guarantee the existence of dicritical singularities of $M$ in presence of invariant compact complex curves by $\mc{F}$. 
\begin{secondtheorem}\label{second}
Let $\mc{F}$ be a holomorphic foliation on a compact complex manifold $X$ of complex dimension two tangent to an irreducible real-analytic Levi-flat hypersurface $M \subset X$. Suppose that the self-intersection $C\cdot C>0$, where $C\subset M$ is an irreducible compact complex curve invariant by $\mc{F}$, then there exists a dicritical singularity $p \in \sing (\mc{F})\cap C$ such that $\mc
{F}$ has a non-constant meromorphic first integral at $p$.
\end{secondtheorem}
\par To prove Theorem \ref{second} we use the Camacho-Sad index (cf. \cite{CS}) and a result of Cerveau-Lins Neto (see Theorem \ref{lins-cerveau}).  
\par We organize the paper as follows: in section \ref{indices} we review some definitions and results about indices of holomorphic foliations at singular points. In section \ref{Existence} we give the proof of Theorem \ref{main_theorem}. Section \ref{application} is devoted to show some applications of Theorem \ref{main_theorem}.
In section \ref{dicritical} we provide the proof of Theorem \ref{second}. Finally, in section \ref{examples_paper} we show some examples which show the importance of the hypotheses in theorems \ref{main_theorem} and \ref{second}. 
\section{Indices of holomorphic foliations}\label{indices}
\par In this section we state two important results on indices of holomorphic foliations: Camacho-Sad index \cite{CS} and 
the Baum-Bott index \cite{baum}. The first one concerns the computation of $C\cdot C$, where $C\subset X$ is an invariant compact curve by $\mathcal{F}$ and the second one concerns the computation of $c^{2}_1(N_{\mc{F}})$. More references for these index theorems can be found in \cite[Chapter V]{suwa}, see also \cite{birational} and \cite{index}.
\par First, we recall the definition of meromorphic and holomorphic first integral for holomorphic foliations. Let $\mc{F}$ be a singular holomorphic foliation on $X$. Recall that $\mathcal{F}$ admit a \textit{meromorphic} (\textit{holomorphic}) first integral at $p\in X$, if there exists a neighborhood $U$ of $p$ and  
a \textit{meromorphic} (\textit{holomorphic}) function $h$ defined in $U$ such that its indeterminacy (zeros)
set is contained in $\sing(\mc{F})\cap U$ and its level curves contain the leaves of $\mathcal{F}$ in $U$.

\subsection{Camacho-Sad index}
\par Let us consider a separatrix $C$ at $p\in X$. Let $f$ be a holomorphic function on a neighborhood of $p$ and defining $C =\{f=0\}$. 
We may assume that $f$ is reduced, i.e. $df\neq 0$ outside $p$. Then \cite{lins}, \cite{suwa} there are functions $g$, $k$ and a 1-form $\eta$ on a neighborhood of $p$ such that
$$g\omega=kdf+f\eta$$
and moreover $k$ and $f$ are prime, i.e. $k\neq 0$ on $C^{*}=C\setminus\{p\}$. 
The Camacho-Sad index \cite{CS} is defined as 
$$\cs(\mc{F},C,p)=-\frac{1}{2\pi i}\int_{\partial{C}}\frac{1}{k}\eta,$$
where $\partial{C}=C\cap S^{3}$ and $S^3$ is a small sphere around $p$; $\partial{C}$ is oriented as  a boundary of $S^{3}\cap B^4$, with $B^4$ a ball containing $p$. 
\par If $C\subset X$ is a compact complex curve invariant by $\mc{F}$, one has the formula due by Camacho-Sad. 
\begin{theorem}\cite[Camacho-Sad]{CS}\label{CS}
$$\sum_{p\in\sing(\mc{F})\cap C}\cs(\mc{F},C,p)=C\cdot C.$$
\end{theorem}

\subsection{Baum-Bott index}
Let $\mc{F}$ be a holomorphic foliation with isolated singularities on $X$. Let $p\in X$ be a singular point of $\mc{F}$; near $p$ the foliation is given by a holomorphic  vector field $$v=F(x,y)\frac{\partial}{\partial{x}}+G(x,y)\frac{\partial}{\partial{y}}$$ or by a holomorphic 1-form $$\omega=F(x,y)dy-G(x,y)dx.$$  
\par Let $J(x,y)$ be the Jacobian matrix of $(F,G)$ at $(x,y)$, then following \cite{baum} one can define the Baum-Bott index at $p\in\sing(\mc{F})$ as
$$\bb(\mc{F},p)=\res_{p}\Big\{\frac{(\tr J)^2}{F\cdot G}dx\wedge dy\Big\}.$$
The Baum-Bott index depended only on the conjugacy class of the germ of $\mc{F}$ at $p$. For example when the singularity $p$ is non-degenerate then $$\bb(\mc{F},p)=\frac{(\tr J(p))^2}{\det J(p)}=\frac{\lambda_1}{\lambda_2}+\frac{\lambda_2}{\lambda_1}+2,$$ where $\lambda_1$ and $\lambda_2$ are the eigenvalues of the linear part $Dv(p)$ of $v$ at $p$. The set of separatrices $S$ through $p$ is formed by two transversal branches $C_1$ and $C_2$, both of them analytic. 
We note also that $$\cs(\mc{F},S,p)=\cs(\mc{F},C_1,p)+\cs(\mc{F},C_2,p)+2[C_1,C_2]_p=\frac{\lambda_1}{\lambda_2}+\frac{\lambda_2} {\lambda_1}+2,$$ where $[C_1,C_2]_{p}$ denotes the intersections number between the curves $C_1$ and $C_2$ at $p$. We remark that $\cs(\mc{F},C_i,p)$ are computed over $\partial{C}_{i}$ for each $i=1,2$ respectively. 
Thus $\bb(\mc{F},p)=\cs(\mc{F},S,p)$. Of course this remains valid for generalized curve foliations with non-dicritical singularities.
\begin{theorem}\cite[Brunella]{index}\label{brunella_index}
Let $\mc{F}$ be a non-dicritical germ of holomorphic foliation at $0\in\mathbb{C}^2$ and let $S$ be the union of all its separatrices. If $\mc{F}$ is a generalized curve foliation, then 
\begin{eqnarray*}
\bb(\mc{F},0)&=&\cs(\mc{F},S,0).
\end{eqnarray*}
\end{theorem}
\par We recall that the foliation, induced by $v$, is said to be \textit{generalized curve} at $0\in\mathbb{C}^2$ if there are no saddle-nodes in its reduction of singularities. It is easily seen that Theorem \ref{brunella_index} is false for saddle-nodes singularities. 

 The following formula will used in the proof of Theorem \ref{main_theorem}.
\begin{theorem}\cite[Baum-Bott]{baum}\label{baum_bott}
$$\sum_{p\in\sing(\mc{F})}\bb(\mc{F},p)=c^{2}_1(N_{\mc{F}}).$$
\end{theorem}


\section{Existence of dicritical singularities of Levi-flat hypersurfaces}\label{Existence}

\par In order to prove Theorem \ref{main_theorem} we need the following results.
\begin{theorem}[Cerveau-Lins Neto \cite{alcides}]\label{lins-cerveau}
 Let $\mathcal{F}$ be a germ at $0\in\mathbb{C}^{n}$, $n\geq{2}$, of codimension one holomoprhic foliation tangent to a germ of an irreducible real-analytic hypersurface $M$. Then $\mathcal{F}$ has a non-constant meromorphic first integral. In the case of dimension two we can precise more:
\begin{enumerate}
\item If $\mc{F}$ is dicritical then it has a non-constant meromorphic first integral.
\item If $\mc{F}$ is non-dicritical then it has a non-constant holomorphic first integral.
\end{enumerate}
\end{theorem}
Now we prove the following lemma. 
\begin{lemma}\label{baum-bott_signo}
Let $\mc{F}$ be a germ of a non-dicritical holomorphic foliation at $0\in\mathbb{C}^2,$ tangent to a germ of an irreducible real-analytic Levi-flat hypersurface $M$ at $0\in\mathbb{C}^2$. Then the Baum-Bott index satisfies 
$\bb(\mc{F},0)\leq 0.$
\end{lemma}
\begin{proof}
Since $\mc{F}$ is non-dicritical at $0\in\mathbb{C}^2$, we have $\mc{F}$ has a non-constant holomorphic first integral $g\in\mathcal{O}_2$ by Theorem \ref{lins-cerveau}. In particular, $\mc{F}$ is a generalized curve foliation and Theorem \ref{brunella_index} implies that 
$$\bb(\mc{F},0)=\cs(\mc{F},S,0),$$
where $S$ is the union of all separatrices of $\mc{F}$ at $0\in\mathbb{C}^2$. To prove the lemma we need calculate  $CS(\mc{F},S,0)$. In fact, as $dg\wedge\omega=0$, then $dg=h\omega$, where $h\in\mathcal{O}_2$. Moreover, if $g=g^{\ell_1}_1\cdots g^{\ell_k}_{k}$, we get $S=\displaystyle\bigcup^{k}_{j=1} C_{j}$ with $C_j=\{g_j=0\}$ and 
$$\sum^{k}_{j=1}\ell_jg_1\cdots\widehat{g_j}\cdots g_k dg_j=h_1\omega,$$
where $h_1=\frac{h}{g^{\ell_1-1}_1\cdots g^{\ell_k-1}_k}$. Hence $$\bb(\mc{F},0)=\cs(\mc{F},S,0)=-\sum_{1\leq i< j\leq k}\frac{(\ell_i-\ell_j)^2}{\ell_i\ell_j}[g_i,g_j]_{0}\leq 0,$$
here $[g_i,g_j]_{0}$ denotes the number of intersections between $C_i$ and $C_j$ at $0\in\mathbb{C}^2$.
\end{proof}
To continue we prove Theorem \ref{main_theorem}.
\subsection{Proof of Theorem \ref{main_theorem}}
We use the Theorem \ref{baum_bott} and Theorem \ref{lins-cerveau} to prove that $\mc{F}$ has a dicritical singularity in $X$. In fact, suppose by contradiction that $\sing(\mc{F})$ consists only of non-dicritical singularities. Take any point $p\in\sing (\mc{F})$ and let $U$ be a small neighborhood of $p$ in $X$  such that $\mc{F}$ is represented by a holomorphic 1-form $\omega$ on $U$ and $p$ is an isolated singularity of $\omega$. Since $\mc{F}$ and $M$ are tangent in $U$ and $p\in M$, we have $\mc{F}|_{U}$ admits a holomorphic first integral $g\in\mc{O}(U)$, that is, $\omega\wedge dg=0$ on $U$, by Theorem \ref{lins-cerveau}. 
\par Applying Lemma \ref{baum-bott_signo}, we get $\bb(\mc{F},p)\leq 0$, for any $p\in\sing(\mc{F})$. But Baum-Bott's formula implies that
$$\sum_{p\in\sing(\mc{F})}\bb(\mc{F},p)=c^{2}_{1}(N_{\mc{F}})>0,$$
which is absurd. Therefore, there exists a dicritical singularity $p$ of $\mc{F}$. Applying again Theorem \ref{lins-cerveau}, we obtain a non-constant meromorphic first integral for $\mc{F}$ in a neighborhood of $p$.

\section{Applications of Theorem \ref{main_theorem}}\label{application}
\par First we apply Theorem \ref{main_theorem} to holomorphic foliations of $\mathbb{P}^2$ with only non-dicritical singularities.
\begin{corollary} 
Let $\mc{F}$ be a holomorphic foliation on $\mathbb{P}^2$ with only non-dicritical singularities. Then there are no singular real-analytic Levi-flat hypersurfaces in $\mathbb{P}^{2}$ tangent to $\mc{F}$ that contain $\sing(\mathcal{F})$. 
\end{corollary}
\begin{proof}
Suppose by contradiction that $\mc{F}$ is tangent to a singular real-analytic Levi-flat hypersurface $M\subset\mathbb{P}^2$ such that $\sing(\mc{F})\subset M$. Since $N_{\mc{F}}=\mathcal{O}_{\mathbb{P}^2}(d+2)$, where $d$ is a positive integer, one has $c^{2}_{1}(N_{\mc{F}})=(d+2)^{2}>0$. Therefore $\mc{F}$ has a dicritical singularity by Theorem \ref{main_theorem}. But it is absurd with the assumption.  
\end{proof}
\par The Jouanolou foliation $\mathcal{J}_d$ of degree $d$ on $\mathbb{P}^2$, is given in affine coordinates $(x,y)\in\mathbb{C}^2$ by 
$$\omega_d=(y^{d}-x^{d+1})dy-(1-x^{d}y)dx.$$
It is well known that $\mathcal{J}_d$ belongs to a holomorphic foliations class of $\mathbb{P}^2$ without algebraic solutions, this means, $\mathcal{J}_d$ does not admit invariant algebraic curves, see for instance \cite{lins}.  
\par On the other hand, we know that $\mathcal{J}_d$ is a foliation with only non-dicritical singularities, because $\mathcal{J}_d$ has $d^{2}+d+1$ singularities and for each singularity passing only two analytic separatrices. Hence the above corollary shows that $\mathcal{J}_d$ does not tangent to any singular real-analytic Levi-flat hypersurface of $\mathbb{P}^2$.
\par Now we apply Theorem \ref{main_theorem} to holomorphic foliations of degree 2 on $\mathbb{P}^2$ with only one singularity. 
\begin{corollary}
Let $\mc{F}$ be a holomorphic foliation of degree 2 on $\mathbb{P}^2$ with a unique singular point $p$. Suppose that $\mc{F}$ is tangent to a singular real-analytic Levi-flat hypersurface $M\subset\mathbb{P}^2$ and $p\in M$. Then, up to automorphism,  $\mc{F}$ is given in affine coordinates $(x,y)\in\mathbb{C}^2$ by the 1-form
$$\omega=x^2dx+y^2(xdy-ydx).$$
Moreover, let $[x:y:z]$ be the homogenous coordinates of $\mathbb{P}^2$, then  $R=\frac{y^3-3x^2z}{3x^3}$ is a rational first integral for $\mc{F}$. 

\end{corollary}
\begin{proof}
Let us assume that $\mc{F}$ has a unique singularity, say $p$. Applying Theorem \ref{main_theorem}  we have $\mc{F}$ has a non-constant meromorphic first integral $f/g$ in a neighborhood of $p$. Now since $\mc{F}$ has degree 2, one can apply the main theorem of \cite{deserti} which asserts that, up automorphism, $\mc{F}$ is given in affine coordinates $(x,y)\in\mathbb{C}^2$ by one of the following types:
\begin{enumerate}
\item $\omega_1=x^2dx+y^2(xdy-ydx)$;
\item $\omega_2=x^2dx+(x+y^2)(xdy-ydx)$;
\item $\omega_3=xydx+(x^2+y^2)(xdy-ydx)$;
\item $\omega_4=(x+y^2-x^2y)dy+x(x+y^2)dx$.
\end{enumerate}
The foliations given in (2) and (3) have first integral respectively 
$$\left(2+\frac{1}{x}+2\left(\frac{y}{x}\right)+\left(\frac{y}{x}\right)^2\right)\exp\left(-\frac{y}{x}\right)\,\,\,\text{and}\,\,\,\,\left(\frac{y}{x}\right)\exp\left(\frac{1}{2}\left(\frac{y}{x}\right)^2-\frac{1}{x}\right),$$ 
 and the foliation in (4) has no meromorphic first integral, see \cite{deserti}. Then $\mc{F}$ can only be the foliation induced by  
 $$\omega_1=x^2dx+y^2(xdy-ydx).$$
It is easy to check that $R(x,y,z)=\frac{y^3-3x^2z}{3x^3}$ is a rational first integral for $\mc{F}$. 
 \end{proof}
 \par An \textit{algebraizable singularity} is a germ of a singular holomorphic foliation which can be defined in some appropriated local chart by a differential equation with algebraic coefficients. It is proved in \cite{genzmer} the existence of countable many classes of saddle-node singularities which are not algebraizable. In \cite{casale}, Casale studied \textit{simple dicritical singularities}, these singularities are those that become nonsingular after one blow-up and such that a unique leaf is tangent to the exceptional divisor with tangency order of one. He show that a simple dicritical singularity with meromorphic first integral is algebraizable (cf. \cite[Theorem 1]{casale}). 
 \begin{theorem}[Casale]\label{Casale_theorem}
 If $\mathcal{F}$ is a simple dicritical foliation at $0\in\mathbb{C}^2$ with a meromorphic first integral then there exist an algebraic surface $S$, a rational function $H$ on $S$ and a point $p\in S$ such that $\mathcal{F}$ is biholomorphic to the foliation given by the level curves of $H$ in a neighborhood of $p$.
 \end{theorem}
 \par Similarly, our aim is give a result of algebraization of singularities for real-analytic Levi-flat hypersurfaces in compact complex manifolds of complex dimension two. 
\begin{corollary}
Let $\mc{F}$ be a holomorphic foliation on a compact complex manifold $X$ of complex dimension two tangent to an irreducible real-analytic Levi-flat hypersurface $M \subset X$ such that $\sing(\mc{F})\subset M$. Suppose that $c^2_1(N_{\mc{F}})>0$ and $\mc{F}$ has only a unique simple dicritical singularity $p\in X$. Then there exists an algebraic surface $V$, a rational
function $H$ on $V$ and a point $q\in V$ such that the germ of $M$  at $p$ is biholomorphic to a semialgebraic Levi-flat hypersurface $M'\subset V$ in a neighborhood of $q$.
\end{corollary}
\begin{proof}
Since $p\in X$ is the unique singularity of $\mc{F}$, we have $\mc{F}$  has a non-constant meromorphic first integral in a neighborhood of $p$ by Theorem \ref{main_theorem}. Then it is follows from Theorem \ref{Casale_theorem} that there exist an algebraic surface $V$, a rational function $H$ on $V$ and a point $q\in V$ such that $\mc{F}$ is biholomorphic to the foliation given by level curves of $H$ in a neighborhood of $q$. According to \cite[Lemma 5.2]{lebl}, there exists a real-algebraic subvariety Levi-flat $N\subset V$ (of real dimension three) such that $M$ is biholomorphic to a subset $M'\subset N$. Thus $M'$ is semialgebraic.
\end{proof}
\section{Dicritical singularities of Levi-flat hypersurfaces in presence of compact leaves}\label{dicritical}
We use Camacho-Sad's formula (Theorem \ref{CS}) and Theorem \ref{lins-cerveau} to prove the following result.

\begin{theorem}\label{variotional}
Let $\mc{F}$ be a holomorphic foliation on a compact complex manifold $X$ of complex dimension two tangent to an irreducible real-analytic Levi-flat hypersurface $M \subset X$. Suppose that $C\cdot C>0$, where $C\subset M$ is an irreducible compact complex curve invariant by $\mc{F}$, then there exists a dicritical singularity $p \in \sing (\mc{F})\cap C$ such that $\mc
{F}$ has a non-constant meromorphic first integral at $p$.
\end{theorem}
\begin{proof}
Suppose by contradiction that all the singularities of $\mc{F}$ over $C$ are non-dicritical. 
Take any point $q\in\sing (\mc{F})\cap C$ and let $U$ be a neighborhood of $q$ in $X$  such that $C\cap U=\{f=0\}$, $\mc{F}$ is represented by a holomorphic 1-form $\omega$ on $U$ and $q$ is an isolated singularity of $\omega$. Since $\mc{F}$ and $M$ are tangent in $U$, we have $\mc{F}|_{U}$ admits a holomorphic first integral $g\in\mc{O}(U)$, that is, $\omega\wedge dg=0$ on $U$, by Theorem \ref{lins-cerveau}. Then $dg=h\omega$, where $h\in\mc{O}(U)$. If 
$g=g^{\ell_1}_1\cdots g^{\ell_k}_{k}$, we get $f=g_i$ for some $i$ and 
$$\sum^{k}_{j=1}\ell_jg_1\cdots\widehat{g_j}\cdots g_k dg_j=h_1\omega,$$
where $h_1=\frac{h}{g^{\ell_1-1}_1\cdots g^{\ell_k-1}_k}$. It follows that, 
$$CS(\mc{F},C,q)=-\sum_{i\neq j}\frac{\ell_{j}}{\ell_i}[g_i,g_j]_q$$
where $[g_{i},g_j]_{q}$ denotes the intersections number between the curves $\{g_{i}=0\}$ and $\{g_{j}=0\}$ at $q$. In particular, 
$$CS(\mc{F},C,q)\leq 0\,\,\,\,\,\text{for any}\,\,q\in\sing (\mc{F})\cap C. $$
It follows from Theorem \ref{CS} that 
$$\sum_{q\in\sing{\mc{F}}\cap C}CS(\mc{F},C,q)=C\cdot C.$$
But this is a contradiction with $C\cdot C> 0.$
Now since $p\in M$, we can apply again Theorem \ref{lins-cerveau} to find a meromorphic first integral for $\mc{F}$ in a neighborhood of $p$. 
\end{proof}

\section{Examples}\label{examples_paper}
First we give an example where the hypotheses of Theorem \ref{main_theorem} are satisfied.  
\begin{example}
The canonical local example of a real-analytic Levi-flat hypersurface in $\mathbb{C}^2$ is given by $\im z_1=0$. This hypersurface can be extended  to all $\mathbb{P}^2$ given by 
$$M=\{[Z_1:Z_2:Z_3]\in\mathbb{P}^2: Z_1\bar{Z}_3-\bar{Z}_1Z_3=0\}.$$
Moreover, this hypersurface is tangent to holomorphic foliation $\mathcal{F}$ given by the level of rational function $Z_1/Z_3$. Note also that $\mathcal{F}$ has degree $0$ and therefore $c^2_1(N_\mathcal{F})=c^2_1(\mathcal{O}_{\mathbb{P}^2}(2))=4$. The foliation $\mathcal{F}$ has a dicritical singularity at $[0:1:0]\in M$.
\end{example}
We now give two examples where Theorem \ref{main_theorem} is false. 
\begin{example}
Consider the Hopf surface $X=(\mathbb{C}^2-\{0\})/\Gamma_{a,b}$ induced by the infinite cyclic subgroup of $GL(2,\mathbb{C})$ generated by the transformation $(z_1,z_2)\mapsto(az_1,bz_2)$ with $|a|,|b|>1$. Levenberg-Yamaguchi \cite{Levenberg} proved that the domain $$D=\{(z_1,z_2)\cdot\Gamma_{a,b}|\,z_1\in\mathbb{C}, \im\,\,z_2>0\}$$ in $X$ with $b\in\mathbb{R}$ is Stein. Furthermore it is bounded by a real-analytic Levi-flat hypersurface 
$$M=\{(z_1,z_2)\cdot\Gamma_{a,b}|\,z_1\in\mathbb{C}, \im\,\,z_2=0\}.$$ 
It is clear that the levels of the holomorphic function $f(z_1,z_2)=z_2$ on $\mathbb{C}^2-\{0\}$ defines a holomorphic foliation $\mc{F}$ on $X$ tangent to $M$. Since any line bundle on $X$ is flat \cite{mall} we have $c^2_1(N_\mc{F})=0$ and $M$ has no singularities in $X$.
\end{example}

\begin{example}
Let $X=\mathbb{P}^1\times\mathbb{P}^1$ and let $\mc{F}$ be the foliation given by the vertical fibration on $X$.  
Now let $\pi:X\to\mathbb{P}^1$ be the projection on the first coordinate and let $M=\pi^{-1}(\gamma)$, where $\gamma$ is a real-analytic embedded loop in $\mathbb{P}^1$. Take a fiber $F\subset M$. We have $F$ is isomorphic to $\mathbb{P}^1$ with $F^2=0$.
Clearly $M$ is a Levi-flat hypersurface in $X$ tangent to $\mc{F}$ and it has no dicritical singularities in $F$.  
\end{example}


\vskip 0.2 in

\noindent{\it\bf Acknowledgments.--}
The authors wishes to express his thanks to Alcides Lins Neto (IMPA) for several helpful comments during the preparation of the paper. Also, we would like to thank the referee for suggestions and pointing out corrections.


\begin{thebibliography}{99}
\bibitem{baum}
P. Baum and R. Bott: Singularities of holomorphic foliations. J. Differential Geom. 7 (1972), 279-432.
\bibitem{burns}  D. Burns, X. Gong:
   Singular Levi-flat real analytic hypersurfaces.
    Amer. J. Math. 121, no. 1, $(1999)$, 23-53.
\bibitem{birational}
M. Brunella: Birational geometry of foliations, IMPA Monographs 1, doi 10.1007/978-3-319-14310-1\_1
\bibitem{brunella} M. Brunella: Singular Levi-flat hypersurfaces and codimension one foliations. Ann. Sc. Norm. Super. Pisa Cl. Sci. (5) vol. VI, no. 4 $(2007)$, 661-672.
\bibitem{index} M. Brunella: Some remarks on indices of holomorphic vector fields. Publicacions Matem\`atiques, vol. 41, no. 2, $(1997)$, 527-544.
\bibitem{CS}
C. Camacho and P. Sad: Invariant varieties through singularities of
vector fields. Annals of Mathematics, vol. 115, no. 3, $(1982)$, 579-595.


\bibitem{casale}
G. Casale: Simple meromorphic functions are algebraic. Bull. Braz. Math. Soc. New Series (2013) 44:309. doi 10.1007/s00574-013-0015-9
\bibitem{deserti}
D. Cerveau, J. Deserti, D. Garba Belko, R. Meziani: G\'eom\'etrie classique de certains feuilletages de degr\'e deux. Bull. Braz. Math. Soc., New Series (2010) 41: 161. https://doi.org/10.1007/s00574-010-0008-x

\bibitem{alcides}  D. Cerveau, A. Lins Neto:
   Local Levi-flat hypersurfaces invariants by a codimension one holomorphic foliation. Amer. J. Math. vol. 133 no. 3, $(2011)$, 677-716. doi:10.1353/ajm.2011.0018
   
 \bibitem{normal} A. Fern\'andez-P\'erez:
On normal forms of singular Levi-flat real analytic hypersurfaces. Bull. Braz. Math. Soc., New Series (2011) 42: 75. https://doi.org/10.1007/s00574-011-0004-9
\bibitem{arturo} A. Fern\'andez-P\'erez: On Levi-flat hypersurfaces tangent to holomorphic webs. Ann. Sci. Toulouse Math (6) 20(3) $(2011)$, 581-597.
\bibitem{generic} A. Fern\'andez-P\'erez: On Levi-flat hypersurfaces with generic real singular set. J. Geom. Anal. (2013) 23: 2020. doi:10.1007/s12220-012-9317-1
\bibitem{arnold} A. Fern\'andez-P\'erez: Normal forms of Levi-flat hypersurfaces with Arnold type singularities. Ann. Sc. Norm. Super. Pisa Cl. Sci. (5) Vol. XIII (2014), 745-774.
\bibitem{libro}
A. Fern\'andez-P\'erez and J. Lebl: Global and local aspects of Levi-flat hypersurfaces. Publ. Mat. IMPA, Rio de Janeiro, 2015. x+65 pp.
\bibitem{genzmer}
Y. Genzmer and L. Teyssier: Existence of non-algebraic singularities of differential equation. J. Differential Equations 248 (2010) 1256-1267.
\bibitem{jouanolou} J.P. Jouanolou: \'Equations de Pfaff Alg\'ebriques. Lecture Notes in Mathematics, vol. 708. Springer, Berlin (1979) (in French).

\bibitem{lebl} J. Lebl: Algebraic Levi-flat hypervarieties in complex projective space. J. Geom. Anal. (2012) 22: 410. https://doi.org/10.1007/s12220-010-9201-9
\bibitem{singularlebl} J. Lebl: Singular set of a Levi-flat hypersurface is Levi-flat. Math. Ann. (2013) 355: 1177. https://doi.org/10.1007/s00208-012-0821-1
\bibitem{Levenberg}
N. Levenberg and H. Yamaguchi: Pseudoconvex domains in the Hopf surface. J. Math. Soc. Japan. Volume 67, Number 1 (2015), 231-273.
\bibitem{lins}
A. Lins Neto: Algebraic solutions of polynomial differential equations and foliations in dimension two, in ``\textit{Holomorphic Dynamics,}'' (Mexico, 1986), Springer, Lectures Notes 1345, 1988, pp. 192-232.
\bibitem{mall}
D. Mall: The cohomology of line bundles on Hopf manifolds. Osaka J. Math. 28 (1991), 999-1015.
\bibitem{pinchuk}
S. Pinchuk, R. Shafikov and A. Sukhov: Dicritical singularities and laminar currents on Levi-flat hypersurfaces. Izv Math, 81 (5) 2017. doi 10.1070/IM8582 
\bibitem{shafikov}
R. Shafikov and A. Sukhov: Germs of singular Levi-flat hypersurfaces and holomorphic foliations. Comment. Math. Helv. 90 (2015), 479-502. 
\bibitem{suwa}
T. Suwa: Indices of vector fields and residues of holomoprhic singular foliations. Hermann (1998).


\end{thebibliography}
\end{document}